\documentclass[12pt]{article}

\usepackage{setspace}
\usepackage{a4wide}

\usepackage{amsthm, amssymb, amstext}
\usepackage{amsmath}
\usepackage{latexsym}
\usepackage[dvips]{graphicx}
\usepackage{comment}
\usepackage{hyperref}
\usepackage{mathtools}
\usepackage{enumerate} 

\usepackage{subcaption}

\usepackage{todonotes}

\def\F{\mathcal{F}}

\newtheorem{theorem}{Theorem} 
\newtheorem{lemma}{Lemma}

\newtheorem{conjecture}{Conjecture}

\newtheorem{problem}{Problem}
\newtheorem{counterex}{Counterexample}

\theoremstyle{definition}
\newtheorem{remark}{Remark}

\usepackage{amsthm}

\usepackage{tikz}
\usetikzlibrary{calc}

\usepackage[T1]{fontenc}
\usepackage{graphics}
\usepackage{lmodern}
\usepackage{anyfontsize}
\usepackage[section]{placeins} 
\usepackage{stmaryrd} 		
\SetSymbolFont{stmry}{bold}{U}{stmry}{m}{n} 
\usepackage{subcaption}
\usepackage{hyperref}
\usepackage{enumerate}
\usepackage[oldvoltagedirection]{circuitikz}

\usepackage{savesym}
\usepackage{amsmath}
\savesymbol{Bbbk}
\usepackage{amssymb} 
\restoresymbol{AMSSYMB}{Bbbk}

\usepackage{appendix}

\usepackage{stmaryrd} 		
\RequirePackage{amsthm}		
\RequirePackage{amsfonts}	
\RequirePackage{bbm} 		
\RequirePackage{verbatim}	
\RequirePackage{mathrsfs}	
\RequirePackage{dsfont}

\usepackage{braket} 
\usepackage{upgreek} 
\usepackage{xspace}

\RequirePackage{ifthen}

\newlength{\leftsize}
\newlength{\rightsize}

\newcommand{\pa}[2][9]{%
	\ifthenelse{#1 = 0}{#2}{}%
	\ifthenelse{#1 = 1}{(#2)}{}%
	\ifthenelse{#1 = 2}{\big(#2\big)}{}%
	\ifthenelse{#1 = 3}{\Big(#2\Big)}{}%
	\ifthenelse{#1 = 4}{\bigg(#2\bigg)}{}%
	\ifthenelse{#1 = 5}{\Bigg(#2\Bigg)}{}%
	\ifthenelse{#1 = 9}{\left(#2\right)}{}%
}

\newcommand{\paf}[2]{\left(\f{#1}{#2}\right)}

\newcommand{\crochets}[2][9]{%
	\ifthenelse{#1 = 0}{#2}{}%
	\ifthenelse{#1 = 1}{[#2]}{}%
	\ifthenelse{#1 = 2}{\big[#2\big]}{}%
	\ifthenelse{#1 = 3}{\Big[#2\Big]}{}%
	\ifthenelse{#1 = 4}{\bigg[#2\bigg]}{}%
	\ifthenelse{#1 = 5}{\Bigg[#2\Bigg]}{}%
	\ifthenelse{#1 = 9}{\left[#2\right]}{}%
}

\newcommand{\accolades}[2][9]{%
	\ifthenelse{#1 = 0}{#2}{}%
	\ifthenelse{#1 = 1}{\{#2\}}{}%
	\ifthenelse{#1 = 2}{\big\{#2\big\}}{}%
	\ifthenelse{#1 = 3}{\Big\{#2\Big\}}{}%
	\ifthenelse{#1 = 4}{\bigg\{#2\bigg\}}{}%
	\ifthenelse{#1 = 5}{\Bigg\{#2\Bigg\}}{}%
	\ifthenelse{#1 = 9}{\left\{#2\right\}}{}%
}

\newcommand{\abs}[2][9]{%
	\ifthenelse{#1 = 0}{#2}{}%
	\ifthenelse{#1 = 1}{|#2|}{}%
	\ifthenelse{#1 = 2}{\big|#2\big|}{}%
	\ifthenelse{#1 = 3}{\Big|#2\Big|}{}%
	\ifthenelse{#1 = 4}{\bigg|#2\bigg|}{}%
	\ifthenelse{#1 = 5}{\Bigg|#2\Bigg|}{}%
	\ifthenelse{#1 = 9}{\left|#2\right|}{}%
}

\newcommand{\partieentiere}[1]{\left\lfloor#1\right\rfloor}
\newcommand{\partieentieresup}[1]{\left\lceil#1\right\rceil}

\newcommand{\f}[2]{	
	\mathchoice%
		{\dfrac{#1}{#2}}
    	{\dfrac{#1}{#2}}
		{\frac{#1}{#2}}
		{\frac{#1}{#2}}
}
\newcommand{\tf}[2]{{#1}/{#2}}

\newcommand{\grando}[2]{\underset{#1}{\text{O}}\left({#2}\right)} 
\newcommand{\petito}[2]{\underset{#1}{\text{o}}({#2})} 

\newcommand{\ceil}[1]{\left\lceil #1 \right\rceil}

\newcommand{\Sum}[2]{\underset{#1}{\overset{#2}{\sum}}} 
\newcommand{\Prod}[2]{\underset{#1}{\overset{#2}{\prod}}} 

\newcommand{\plus}{\mbox{\protect\raisebox{.2mm}{\tiny{\ensuremath{+}}}}}
\newcommand{\moins}{\mbox{\protect\raisebox{.2mm}{\tiny{\ensuremath{-}}}}}
\newcommand{\pinf}{{\plus\ensuremath{\infty}}}
\newcommand{\minf}{{\moins\ensuremath{\infty}}}

\newcommand{\TestGauche}[1]{\ifthenelse{\equal{#1}{}}{\minf}{#1}}
\newcommand{\TestDroite}[1]{\ifthenelse{\equal{#1}{}}{\pinf}{#1}}

\newsavebox{\boite}
\def\debutcom{\begin{lrbox}{\boite}}
\def\fincomg{\end{lrbox}\makebox[0cm][l]{\usebox{\boite}}%
             \hskip\linewidth\hskip-\textwidth}
\def\fincomd{\end{lrbox}\makebox[0cm][r]{\usebox{\boite}}}
\newenvironment{calculs:rcl}[4][2]{%
  	\par\vskip0.1cm\noindent
  	\begin{tabular*}{\linewidth}[t]%
  		{@{}>{\debutcom}l<{\fincomg}@{\extracolsep{\fill}}
		>{$}#2<{$}@{\extracolsep{0pt}}%
		>{$\ }#3<{\ $}%
      		@{\extracolsep{0pt}}>{$}#4<{$}%
      		@{\extracolsep{\fill}}>{\debutcom}r<{\fincomd}@{}}%
	}{%
  	\end{tabular*}\par%
  	}

        {\begin{calculs:rcl}[#1]{r}{c}{l}}%
        {\end{calculs:rcl}}

\if@anglais
	
\else
	
\fi

\newtheorem*{Thm*}{Theorem}

\newtheorem*{Lem*}{Lemma}

\theoremstyle{definition}

\newtheorem*{Def*}{Definition}

\newtheorem*{Ex*}{Example}

\newtheorem*{CEx*}{Counterexample}

\theoremstyle{remark}

\newtheorem*{Rk*}{Remark}

\newtheorem*{Att*}{Warning}

\renewcommand{\emptyset}{\varnothing}

\title{Solution to a problem of Katona on counting cliques of weighted graphs}
\author{Peter Borg\thanks{Department of Mathematics, Faculty of Science, University of Malta, Malta, email: \texttt{peter.borg@um.edu.mt}} \quad Carl Feghali\thanks{Univ Lyon, EnsL, CNRS, LIP, F-69342, Lyon Cedex 07, France, email: \texttt{carl.feghali@ens-lyon.fr} } \quad R\'emi Pellerin\thanks{Univ Lyon, EnsL, CNRS, LIP, F-69342, Lyon Cedex 07, France, email: \texttt{remi.pellerin@ens-lyon.fr} }}

\date{}

\begin{document}
\maketitle

\begin{abstract}
A subset $I$ of the vertex set $V(G)$ of a graph $G$ is called a \emph{$k$-clique independent set of $G$} if no $k$ vertices in $I$ form a $k$-clique of $G$. An independent set is a $2$-clique independent set. Let $\pi_k(G)$ denote the number of $k$-cliques of $G$. For a function $w: V(G) \rightarrow \{0, 1, 2, \dots\}$, let $G(w)$ be the graph obtained from $G$ by replacing each vertex $v$ by a $w(v)$-clique $K^v$ and making each vertex of $K^u$ adjacent to each vertex of $K^v$ for each edge $\{u,v\}$ of $G$. For an integer $m \geq 1$, consider any $w$ with $\sum_{v \in V(G)} w(v) = m$. For $U \subseteq V(G)$, we say that $w$ is \emph{uniform on $U$} if $w(v) = 0$ for each $v \in V(G) \setminus U$ and, for each $u \in U$, $w(u) = \left\lfloor m/|U| \right\rfloor$ or $w(u) = \left\lceil m/|U| \right\rceil$. Katona asked if $\pi_k(G(w))$ is smallest when $w$ is uniform on a largest $k$-clique independent set of $G$. He placed particular emphasis on the Sperner graph $B_n$, given by $V(B_n) = \{X \colon X \subseteq \{1, \dots, n\}\}$ and $E(B_n) = \{\{X,Y\} \colon X \subsetneq Y \in V(B_n)\}$. He provided an affirmative answer for $k = 2$ (and any $G$). We determine graphs for which the answer is negative for every $k \geq 3$. These include $B_n$ for $n \geq 2$. Generalizing Sperner's Theorem and a recent result of Qian, Engel and Xu, we show that $\pi_k(B_n(w))$ is smallest when $w$ is uniform on a largest independent set of $B_n$. We also show that the same holds for complete multipartite graphs and chordal graphs. We show that this is not true of every graph, using a deep result of Bohman on triangle-free graphs.
\end{abstract}

\section{Introduction}

Let $\mathbb{N}$ denote set of positive integers, and let $\mathbb{N}_0$ denote $\{0\} \cup \mathbb{N}$. For $n \in \mathbb{N}_0$, let $[n]$ denote the $n$-set $\{i \in \mathbb{N} \colon i \le n\}$ (note that $[0] = \emptyset$). For a set $X$, let $2^X$ denote the \emph{power set of $X$} ($\{A \colon A \subseteq X\}$) and, for $k \in \mathbb{N}_0$, let $\binom{X}{k}$ denote $\{A \in 2^X \colon |A| = k\}$. 

A family $\mathcal{A}$ of sets is called an \emph{antichain} or a \emph{Sperner family} if $A \nsubseteq B$ for every $A, B \in \mathcal{A}$ with $A \neq B$. A cornerstone in extremal
set theory is Sperner's Theorem \cite{Sperner}, which bounds the size of an antichain. 

\begin{theorem}[Sperner's Theorem \cite{Sperner}]\label{thm:sperner}
	If $\mathcal{A} \subseteq 2^{[n]}$ and $\mathcal{A}$ is an antichain, then
	$$
	|\mathcal{A}| \leq \binom{n}{\lceil \frac{n}{2} \rceil}.
	$$

	\noindent Moreover, equality holds if and only if $\mathcal{A} =
	\binom{[n]}{\partieentiere{\f{n}{2}}}$ or  $\mathcal{A} =
	\binom{[n]}{\partieentieresup{\f{n}{2}}}$.
\end{theorem}

There have been many generalizations, extensions and variants of
Theorem~\ref{thm:sperner}; see, for example,~\cite{DGS, DGKS, Erdos, Kleitman}. Of particular relevance to this paper is a generalization due to Qian, Engel and Xu~\cite{QEX} in which repetition of sets is allowed. A multifamily is a pair $(\mathcal{F},q)$ such that $\mathcal{F}$ is a family and $q$ is a function with domain $\F$ and codomain $\mathbb{N}_0$. A multifamily can be viewed as a family $\F$ such that, for each $F \in \mathcal{F}$, $F$ appears $q(F)$ times.  
Let

$$
\theta(\F,q) = \Sum{F \in \mathcal{F}}{} \binom{q(F)}{2} + \Sum{F, F' \in \mathcal{F} \colon F
\subsetneq F'}{} q(F)q(F').
$$

\begin{theorem}[Qian, Engel and Xu~\cite{QEX}]\label{thm:qian}
	For any $n \geq 1$ and $m \geq 1$, the minimum of $\theta(\F,q)$ over all multifamilies $(\F,q)$ with $\F \subseteq 2^{[n]}$ and $\sum_{F \in \mathcal{F}} q(F) = m$ is attained if
	$\F \in \left\{\binom{[n]}{\lfloor n/2 \rfloor}, \binom{[n]}{\lceil n/2 \rceil} \right\}$ and $q(F) \in \left\{ \left\lfloor m/|\F| \right\rfloor, \left\lceil m/|\F| \right\rceil \right\}$ for each $F \in \F$.
\end{theorem}
\noindent Qian, Engel and Xu actually proved that the result holds for the quantity $\theta(\F,q) + \Sum{F \in \mathcal{F}}{} \binom{q(F)}{2}$ rather than $\theta(\F,q)$ \cite[Theorem~1.1]{QEX}, but Theorem~\ref{thm:qian} follows from this and Theorem~\ref{thm:sperner}.

Recently, Katona~\cite{Katona} obtained a far-reaching generalization of Theorem~\ref{thm:qian}. To be able to state his result, we require a number of definitions.

As usual, we denote the vertex set and the edge set of a graph $G$ by $V(G)$ and $E(G)$, respectively. We take any graph $G$ to be simple, that is, $G = (V(G), E(G))$ with $E(G) \subseteq \binom{V(G)}{2}$. We may represent any edge $\{u,v\}$ by $uv$. The \emph{open neighbourhood} $N_G(v)$ of a vertex $v$ of $G$ is the set of neighbours of $v$, that is, $N_G(v) = \{u \in V(G) \colon uv \in E(G)\}$. 
The \emph{closed neighbourhood} $N_G[v]$ of $v$ is the set $N_G(v) \cup \{v\}$. For $X \subseteq V(G)$, $G[X]$ denotes the subgraph of $G$ \emph{induced by $X$}, that is, $G[X] = (X, E(G) \cap \binom{X}{2})$. 

For a graph $G$ and a (weight) function $w: V(G) \rightarrow \mathbb{N}_0$, let $G(w)$ be the graph obtained from $G$ by replacing each vertex $v$ by a $w(v)$-clique $K^v$ and making each vertex of $K^u$ adjacent to each vertex of $K^v$ for each edge $uv$ of $G$. 
More formally, $G(w)$ is given by
$$V(G(w)) = \{(v,i) \colon v \in V(G), \, i \in [w(v)]\}$$ 
and 
$$\quad E(G(w)) = \bigcup_{v \in V(G)} {\{(v,i) \colon i \in [w(v)]\} \choose 2} \cup \bigcup_{uv \in E(G)} \{(u,i)(v,j) \colon i \in [w(u)], \, j \in [w(v)]\}.$$
%
We call $w$ an \emph{$m$-weighting of $G$}, where $m = \sum_{v \in V(G)} w(v)$. For $k \geq 1$, let $\mathcal{K}_k(G)$ denote the set of vertex sets of $k$-cliques ($k$-vertex complete subgraphs) of $G$, and let $\pi_k(G)$ denote the number of $k$-cliques of $G$. Thus, $\pi_k(G) = |\mathcal{K}_k(G)|$. Observe that the number of edges of $G(w)$ is

$$
\pi_2(G(w)) = \sum_{v \in V(G)} \binom{w(v)}{2} + \sum_{uv \in E(G)} w(u)w(v), 
$$

\noindent and for any $k \geq 1$,

	$$\pi_k(G(w)) = \Sum{i = 1}{k} \quad \Sum{\{v_1, \dots, v_i\} \in \mathcal{K}_i(G) }{}
	\quad
	\Sum{\substack{k_1 + \cdots + k_i = k \\ k_1,\ldots,k_i \geq 1}}{} \quad \Prod{j = 1}{i} \binom{w(v_j)}{k_j}.$$


Given an $m$-weighting $w$ of $G$ and a subset $U$ of $V(G)$, we say that $w$ is \emph{uniform on $U$} if $w(v) = 0$ for each $v \in V(G) \setminus U$ and, for each $u \in U$, $w(u) = \left\lfloor m/|U| \right\rfloor$ or $w(u) = \left\lceil m/|U| \right\rceil$. If $U$ is an independent set of $G$ (that is, $uv \notin E(G)$ for every $u, v \in U$) of maximum size and $w$ is uniform on $U$, then, as in~\cite{Katona}, $w$ is said to be \emph{uniform-$\alpha$}.

For $n \geq 1$, let $B_n$ be the \emph{Sperner graph} given by $V(B_n) = 2^{[n]}$ and $E(G) = \{XY \colon X \subsetneq Y \in 2^{[n]}\}$. Note that Sperner's Theorem gives us the size of a largest independent set of $B_n$. Theorem~\ref{thm:qian} may thus be restated as follows.

\begin{theorem}\label{thm:engel} If $n \geq 1$, $m \geq 1$, $w$ and $w'$ are $m$-weightings of $B_n$, and $w'$ is uniform-$\alpha$, then $\pi_2(B_n(w')) \leq \pi_2(B_n(w))$.
\end{theorem}

\noindent Rather surprisingly, Katona showed that Theorem~\ref{thm:engel} can be generalized to arbitrary graphs. 

\begin{theorem}[Katona \cite{Katona}]\label{thm:katona}
If $m \geq 1$, $w$ and $w'$ are $m$-weightings of a graph $G$, and $w'$ is uniform-$\alpha$, then $\pi_2(G(w')) \leq \pi_2(G(w))$.
\end{theorem}

\noindent He then asked the general question below. For a graph $G$ and $k \geq 1$, we call a subset $I$ of $V(G)$ a \emph{$k$-clique independent set of $G$} if no $k$-element subset of $I$ is a member of $\mathcal{K}_k(G)$ (that is, $\binom{I}{k} \cap \mathcal{K}_k(G) = \emptyset$). Let the size of a largest $k$-clique independent set of $G$ be denoted by $\alpha_k(G)$ and called the \emph{$k$-clique independence number of~$G$}.

\begin{problem}[{\cite[Problem 3]{Katona}}] \label{problem_naif}
Is it true that if $m \geq 1$, $k \geq 2$, $w$ and $w'$ are $m$-weightings of a graph $G$, and $w'$ is uniform on a largest $k$-clique independent set of $G$, then $\pi_k(G(w')) \leq \pi_k(G(w))$?
\end{problem}

\noindent Theorem~\ref{thm:katona} provides a positive answer to
Problem~\ref{problem_naif} for $k = 2$. Unfortunately, if $k \geq 3$, one cannot hope for a positive answer to Problem~\ref{problem_naif}, as shown in the following counterexample and in Section~\ref{sec:proofs} (see Remarks~\ref{rem:sperner_min} and \ref{rem:cmg}). 

\begin{counterex}\label{c-ex} \emph{
Let $G$ be the $3$-vertex path $(\{a, b, c\}, \{ab, bc\})$, let $k \geq 3$, let $w'$ be the $3k$-weighting of $G$ that is uniform on $V(G)$, and let $w$ be the $3k$-weighting of $G$ defined by $w(a) = w'(a) + w'(b)$, $w(b) = 0$, and $w(c) = w'(c)$. Note that $V(G)$ is the largest $k$-clique independent set of $G$. We have}
$$\pi_k(G(w'))  =  \binom{w'(a) + w'(b)}{k} + \binom{w'(b) + w'(c)}{k} - \binom{w'(b)}{k} = 2\binom{2k}{k} - \binom{k}{k},$$
$$\pi_k(G(w)) = \binom{w'(a) + w'(b)}{k} + \binom{w'(c)}{k} =
\binom{2k}{k} + \binom{k}{k},$$

\noindent \emph{and hence $\pi_k(G(w)) < \pi_k(G(w'))$.}
\end{counterex}

As our primary contribution, we completely address Problem~\ref{problem_naif} for Sperner graphs, thereby answering another question of Katona \cite[Problem 3]{Katona}. 

\begin{theorem}\label{thm:sperner_min}
If $n \geq 1$, $m \geq 1$, $k \geq 2$, $w$ and $w'$ are $m$-weightings of $B_n$, and $w'$ is uniform on $\binom{[n]}{\lceil n/2 \rceil}$, then $\pi_k(B_n(w')) \leq \pi_k(B_n(w))$.
\end{theorem}
\noindent Theorem~\ref{thm:sperner_min} generalizes the inequalities in Theorems~\ref{thm:sperner} and \ref{thm:engel}. Indeed, let $\mathcal{A} \subseteq 2^{[n]}$ such that $\mathcal{A}$ is an antichain. Let $w_1$ and $w_2$ be $\left( \binom{n}{\lceil n/2 \rceil} + 1 \right)$-weightings of $B_n$ such that $w_1$ is uniform on $\mathcal{A}$ and $w_2$ is uniform on $\binom{[n]}{\lceil n/2 \rceil}$. Note that $\pi_2(B_n(w_2)) = 1$. Thus, by Theorem~\ref{thm:sperner_min}, $\pi_2(B_n(w_1)) \geq 1$. The inequality in Theorem~\ref{thm:sperner} follows. Now, by Theorem~\ref{thm:sperner}, $w'$ is uniform-$\alpha$. By Theorem~\ref{thm:sperner_min} for $k = 2$, Theorem~\ref{thm:engel} follows. 

\begin{remark} \label{rem:sperner_min} In \cite[Problem 3]{Katona}, Katona placed particular emphasis on solving Problem~\ref{problem_naif} for $G = B_n$. Theorem~\ref{thm:sperner_min} already tells us how to minimize $\pi_k(B_n(w))$. In Section~\ref{sec:proofs}, we show that, furthermore, the answer to Problem~\ref{problem_naif} for $G = B_n$ is negative if $n \geq 2$, $k \geq 3$, and $m \geq k \alpha_k(B_n)$.
\end{remark}

We also address Problem~\ref{problem_naif} for complete multipartite graphs and chordal graphs.

If $I_1, \dots, I_r$ are pairwise disjoint non-empty sets and $G$ is the graph with $V(G) = \bigcup_{i = 1}^r I_i$ and $E(G) = \bigcup_{\{i, j\} \in \binom{[r]}{2}} \{xy \colon x \in I_i, \, y \in I_j\}$, then $G$ is called a \emph{complete multipartite graph}, and $I_1, \dots, I_r$ are called the \emph{maximal partite sets of $G$}.

\begin{theorem}\label{thm:cmg}
If $m \geq 1$, $k \geq 2$, $w$ and $w'$ are $m$-weightings of a complete multipartite graph $G$, and $w'$ is uniform-$\alpha$, then $\pi_k(G(w')) \leq \pi_k(G(w))$.
\end{theorem}

\begin{remark}\label{rem:cmg} In Section~\ref{sec:proofs}, we show that, furthermore, if $G$ is a complete multipartite graph with a maximal partite set $I$ that is larger than the others, $k \geq 3$, and $m \geq k \alpha_k(G)$, then the answer to Problem~\ref{problem_naif} is negative. Note that Counterexample 1 is the special case where the maximal partite sets of $G$ are $\{a,c\}$ and $\{b\}$.
\end{remark}

Let $\mathcal{K}(G)$ denote the set of vertex sets of cliques of $G$; that is, $\mathcal{K}(G) = \bigcup_{k = 1}^{|V(G)|} \mathcal{K}_k(G)$. A graph $G$ is said to be \emph{chordal} if for some sequence $v_1, \dots, v_n$ such that $V(G) = \{v_1, \dots, v_n\}$ and $n = |V(G)|$, $N_G[v_i] \setminus \{v_j \colon j \in [i-1]\} \in \mathcal{K}(G)$ for each $i \in [n]$.

\begin{theorem}\label{thm:chordal}
If $m \geq 1$, $k \geq 2$, $w$ and $w'$ are $m$-weightings of a chordal graph $G$, and $w'$ is uniform-$\alpha$, then $\pi_k(G(w')) \leq \pi_k(G(w))$.
\end{theorem}

In view of Theorems~\ref{thm:sperner_min}--\ref{thm:chordal}, one may wonder whether the minimum of $\pi_k(G(w))$ is always attained when $w$ is uniform-$\alpha$. We show this to be false for $k=3$, using a result on triangle-free graphs (graphs containing no $3$-clique) due to Bohman \cite{Bohman}.

\begin{theorem}\label{thm:con_stable_false}
There exist a graph $G$, a positive integer $m$, and an $m$-weighting $w$ of $G$ such that $\pi_3(G(w)) < \pi_3(G(w'))$ for every uniform-$\alpha$ $m$-weighting $w'$ of $G$.
\end{theorem}

We propose the following conjecture.

\begin{conjecture}\label{con:main}
If $m \geq 1$, $k \geq 2$, and $G$ is a graph, then for some $k$-clique independent set $I$ of $G$ and some $m$-weighting $w'$ of $G$ that is uniform on $I$, $\pi_k(G(w')) \leq \pi_k(G(w))$ for any $m$-weighting $w$ of $G$.
\end{conjecture}

The paper is organized as follows. Section~\ref{sec:basiclemmas} contains basic tools that are used in the proofs of Theorems~\ref{thm:sperner_min}--\ref{thm:chordal}. In Section~\ref{sec:shifting_lemma}, we establish a general weight shifting lemma from which Theorems~\ref{thm:sperner_min}--\ref{thm:chordal} follow, and we prove Theorems~\ref{thm:cmg} and \ref{thm:chordal}. In Section~\ref{sec:proofs}, we prove our main results, given by Theorems~\ref{thm:sperner_min} and \ref{thm:con_stable_false} and the claims in Remarks~\ref{rem:sperner_min} and \ref{rem:cmg}.


\section{Basic lemmas}\label{sec:basiclemmas}

The following known fact is very useful, and we prove it for completeness.

\begin{lemma}\label{lem:cauchy}
If $n \geq 1$, $m \geq 1$, $k \geq 2$, $w$ and $w'$ are $m$-weightings of an $n$-vertex graph $G$ with no edges, and $w'$ is uniform-$\alpha$, then $\pi_k(G(w')) \leq \pi_k(G(w))$. 
\end{lemma}
\begin{remark} \label{rem:cauchy}
Let $k$, $G$, and $w'$ be as in Lemma~\ref{lem:cauchy}. Thus, $w'(v) \in \{\lfloor m/n \rfloor, \lceil m/n \rceil\}$ for each $v \in V(G)$. By the division algorithm, $m = \lfloor m/n \rfloor n + r$ for some $r \in \{0\} \cup [n-1]$. If $r = 0$, then $\lfloor m/n \rfloor = m/n = \lceil m/n \rceil$. Suppose $r \neq 0$. Then, $m/n > \lfloor m/n \rfloor = \lceil m/n \rceil - 1$. Since $\sum_{v \in V(G)} w'(v) = m$, we obtain $|\{v \in V(G) \colon w'(v) = \lfloor m/n \rfloor\}| = n-r$ and $|\{v \in V(G) \colon w'(v) = \lceil m/n \rceil\}| = r$. Therefore, if $w_1$ and $w_2$ are uniform-$\alpha$ $m$-weightings of $G$, then $\pi_k(G(w_1)) = \pi_k(G(w_2))$.
\end{remark}

\begin{proof}[\textbf{Proof of Lemma~\ref{lem:cauchy}}] Let $v_1, v_2 \in V(G)$ such that $w(v_1) = \min\{w(v) \colon v \in V(G)\}$ and $w(v_2) = \max\{w(v) \colon v \in V(G)\}$. Since $\sum_{v \in V(G)} w(v) = m$, we have $w(v_1) \leq m/n$ and $w(v_2) \geq m/n$. Since $w(v) \in \mathbb{N}_0$ for each $v \in V(G)$, $w(v_1) \leq \lfloor m/n \rfloor$ and $w(v_2) \geq \lceil m/n \rceil$. If $w(v_1) = \lfloor m/n \rfloor$ and $w(v_2) = \lceil m/n \rceil$, then $w$ is uniform-$\alpha$, so $\pi_k(G(w')) = \pi_k(G(w))$ by Remark~\ref{rem:cauchy}. Suppose $w(v_1) \neq \lfloor m/n \rfloor$ or $w(v_2) \neq \lceil m/n \rceil$. Then, $w(v_1) \leq \lfloor m/n \rfloor - 1$ or $w(v_2) \geq \lceil m/n \rceil + 1$.

Suppose $w(v_1) \leq \lfloor m/n \rfloor - 1$. Let $w_1$ be the $m$-weighting of $G$ such that $w_1(v_1) = w(v_1)+1$, $w_1(v_2) = w(v_2)-1$, and $w_1(v) = w(v)$ for each $v \in V(G) \setminus \{v_1, v_2\}$. We have
\begin{align}
\pi_k(G(w)) - \pi_k(G(w_1)) &= \binom{w(v_2)}{k} - \binom{w(v_2)-1}{k} + \binom{w(v_1)}{k} - \binom{w(v_1) + 1}{k} \nonumber \\ 
&= \binom{w(v_2) - 1}{k-1} - \binom{w(v_1)}{k-1} \geq 0 \label{perturbation1}
\end{align}
\noindent as $w(v_2) - 1 \geq w(v_1)$ and $k-1 \geq 1$. Thus, $\pi_k(G(w)) \geq \pi_k(G(w_1))$. We apply this procedure until we obtain an $m$-weighting $w_p$ of $G$ such that $\min\{w_p(v) \colon v \in V(G)\} = \lfloor m/n \rfloor$. We have $\pi_k(G(w)) \geq \pi_k(G(w_p))$. Let $v_{p,1}, v_{p,2} \in V(G)$ such that $w_p(v_{p,1}) = \lfloor m/n \rfloor$ and $w_p(v_{p,2}) = \max\{w_p(v) \colon v \in V(G)\}$. If $w_p(v_{p,2}) = \lceil m/n \rceil$, then $w_p$ is uniform-$\alpha$, so $\pi_k(G(w')) = \pi_k(G(w_p))$ by Remark~\ref{rem:cauchy}. Suppose $w_p(v_{p,2}) \neq \lceil m/n \rceil$. Since $w_p(v_{p,2}) \geq \lceil m/n \rceil$, we obtain $w_p(v_{p,2}) \geq \lceil m/n \rceil + 1$. Since $\sum_{v \in V(G)} w_p(v) = m$, $m/n > \lfloor m/n \rfloor = w_p(u)$ for some $u \in V(G) \setminus \{v_{p,1}\}$. Let $w_{p+1}$ be the $m$-weighting of $G$ such that $w_{p+1}(u) = w_{p}(u)+1$, $w_{p+1}(v_{p,2}) = w_p(v_{p,2})-1$, and $w_{p+1}(v) = w_p(v)$ for each $v \in V(G) \setminus \{u, v_{p,2}\}$. As in (\ref{perturbation1}), we obtain $\pi_k(G(w_p)) \geq \pi_k(G(w_{p+1}))$. We apply this procedure until we obtain an $m$-weighting $w_q$ of $G$ such that $\max\{w_q(v) \colon v \in V(G)\} = \lceil m/n \rceil$. Since $w_q(v_{p,1}) = \lfloor m/n \rfloor = \min\{w_q(v) \colon v \in V(G)\}$, $w_q$ is uniform-$\alpha$, so $\pi_k(G(w')) = \pi_k(G(w_q))$ by Remark~\ref{rem:cauchy}.

If $w(v_2) \geq \lceil m/n \rceil + 1$, then $\pi_k(G(w')) \leq \pi_k(G(w))$ by a similar argument.
\end{proof}

\begin{lemma} \label{cor:cauchy} If $n \geq 1$, $m \geq 1$, $k \geq 2$, $w$ and $w'$ are $m$-weightings of an $n$-vertex graph $G$, $I$ and $I'$ are independent sets of $G$ with $|I| \leq |I'|$, $w(v) = 0$ for each $v \in V(G) \setminus I$, and $w'$ is uniform on $I'$, then $\pi_k(G(w')) \leq \pi_k(G(w))$.
\end{lemma}
\begin{proof} 
Let $u_1, \dots, u_r$ be the distinct vertices in $I$. Let $v_1, \dots, v_s$ be the distinct vertices in $I'$. Then, $r \leq s$. Let $H = G[I']$. Let $w_H$ be the $m$-weighting of $H$ such that $w_H(v_i) = w(u_i)$ for each $i \in [r]$ and $w_H(v_j) = 0$ for each $j \in [s] \setminus [r]$. Let $w_H'$ be the uniform-$\alpha$ $m$-weighting of $H$ such that $w_H'(v_i) = w'(v_i)$ for each $i \in [s]$. By Lemma~\ref{lem:cauchy}, $\pi_k(H(w_H')) \leq \pi_k(H(w_H))$. Since $\pi_k(H(w_H')) = \pi_k(G(w'))$ and $ \pi_k(H(w_H)) = \pi_k(G(w))$, the result follows. 
\end{proof}


\section{A weight shifting lemma}\label{sec:shifting_lemma}

In the proof of Theorem~\ref{thm:katona}, Katona defined the following weight shifting operation along an edge. For a graph $G$, an $m$-weighting $w$ of $G$, and an edge $ab$ of $G$, let $w_{ab}$ be the $m$-weighting of $G$ given by
$$
w_{ab}(v)=\left\{\begin{array}{ll}0 \quad & \textrm{if $v =  a$},\\ 
w(b) + w(a) \quad & \textrm{if $v = b$},\\
w(v) \quad & \textrm{otherwise}\end{array}\right.
$$
\noindent for each $v \in V(G)$. It was proved in~\cite{Katona} that $\pi_2(G(w_{ab})) \leq \pi_2(G(w))$ or $\pi_2(G(w_{ba})) \leq \pi_2(G(w))$. Thus, by applying the weight shift operation repeatedly, one arrives at an $m$-weighting $w'$ of $G$ such that $\pi_2(G(w')) \leq \pi_2(G(w))$ and $w'$ is non-zero only on an independent set. Theorem~\ref{thm:katona} follows from this.

\begin{remark}
Unfortunately, for $k \geq 3$, Katona's shifting technique does not always decrease $\pi_k(G(w))$ or leave it unchanged. For instance, for $k=3$, consider the illustration in Figure~\ref{fig:shift_edge}. If we start with the $m$-weighting in Figure~\ref{fig:shift_edge}(a), then each shift produces a larger number of triangles, as demonstrated in Figures~\ref{fig:shift_edge}(b) and~\ref{fig:shift_edge}(c). 

\begin{figure}[ht]
\centering
\begin{subfigure}{0.3\textwidth}
\centering
\scalebox{0.7}{\begin{tikzpicture}
	\renewcommand{\theta}{-30}
	\node[draw,circle] (0) at (0 + \theta:1) {$1$};
	\node[draw,circle] (1) at (360 / 3 + \theta:1) {$1$};
	\node[draw,circle] (2) at (2 * 360 / 3 + \theta:1) {$1$};
	\node[draw,circle] (3) at (0 + \theta:3) {$1$};
	\node[draw,circle] (4) at (360 / 3 + \theta:3) {$1$};
	\node[draw,circle] (5) at (2 * 360 / 3 + \theta:3) {$1$};

	\draw[-] (0) to (1) to (2) to (0);
	\draw[-] (3) to (4) to (5) to (3);
	\draw[-] (0) -- (3);
	\draw[-] (1) -- (4);
	\draw[-] (2) -- (5);
\end{tikzpicture}}
\caption{$\pi_3(G(w)) = 2$\label{fig:shift_edge:a}}
\end{subfigure}
\hfill
\begin{subfigure}{0.3\textwidth}
\centering
\scalebox{0.7}{\begin{tikzpicture}
	\renewcommand{\theta}{-30}
	\node[draw,circle] (0) at (0 + \theta:1) {$2$};
	\node[draw,circle] (1) at (360 / 3 + \theta:1) {$0$};
	\node[draw,circle] (2) at (2 * 360 / 3 + \theta:1) {$1$};
	\node[draw,circle] (3) at (0 + \theta:3) {$1$};
	\node[draw,circle] (4) at (360 / 3 + \theta:3) {$1$};
	\node[draw,circle] (5) at (2 * 360 / 3 + \theta:3) {$1$};

	\draw[-] (0) to (1) to (2) to (0);
	\draw[-] (3) to (4) to (5) to (3);
	\draw[-] (0) -- (3);
	\draw[-] (1) -- (4);
	\draw[-] (2) -- (5);
\end{tikzpicture}}
\caption{$\pi_3(G(w)) = 3$}\label{fig:shift_edge:b}
\end{subfigure}
\hfill
\begin{subfigure}{0.3\textwidth}
\centering
\scalebox{0.7}{\begin{tikzpicture}
	\renewcommand{\theta}{-30}
	\node[draw,circle] (0) at (0 + \theta:1) {$1$};
	\node[draw,circle] (1) at (360 / 3 + \theta:1) {$0$};
	\node[draw,circle] (2) at (2 * 360 / 3 + \theta:1) {$1$};
	\node[draw,circle] (3) at (0 + \theta:3) {$1$};
	\node[draw,circle] (4) at (360 / 3 + \theta:3) {$2$};
	\node[draw,circle] (5) at (2 * 360 / 3 + \theta:3) {$1$};

	\draw[-] (0) to (1) to (2) to (0);
	\draw[-] (3) to (4) to (5) to (3);
	\draw[-] (0) -- (3);
	\draw[-] (1) -- (4);
	\draw[-] (2) -- (5);
\end{tikzpicture}}
\caption{$\pi_3(G(w)) = 4$}\label{fig:shift_edge:c}
\end{subfigure}
\caption{A weight shift along an edge
\label{fig:shift_edge}}
\end{figure}
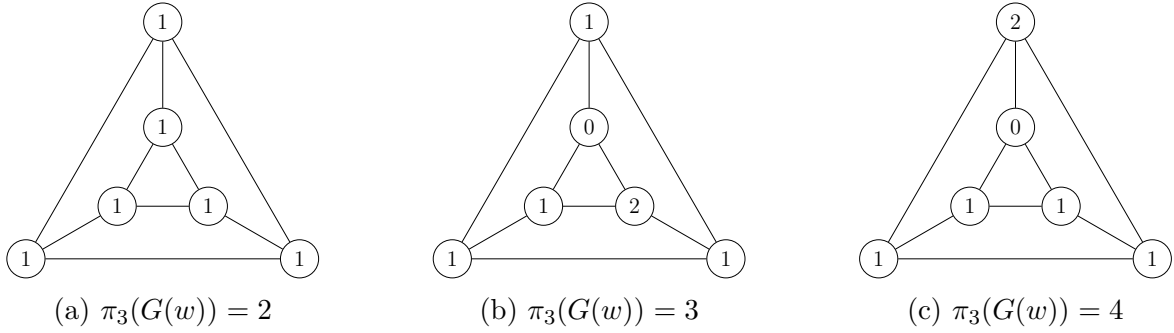
\end{remark}
 
In view of the remark above, new ideas would therefore be needed to address Problem~\ref{problem_naif}. Still, instead of shifting along one edge at a time, is shifting simultaneously along many edges at a time conceivable? In our main lemma, we show that this is indeed the case, provided a number of conditions are satisfied. 

\begin{lemma}\label{lem:main}
If $m \geq 1$, $k \geq 2$, $w$ is an $m$-weighting of a graph $G$, $A = \{a_1,\ldots,a_r\}$ and $B = \{b_1,\ldots,b_r\}$ are disjoint $r$-element subsets of $V(G)$ such that
	\begin{enumerate}
		\item $a_i b_i \in E(G)$ for each $i \in [r]$,
		\item $B$ is an independent set of $G$, and  
		\item $N_G(b_i) \setminus \pa{A \cup \{v \in V(G) \colon w(v) = 0\}} \subseteq N_G(a_i)$ for each $i \in [r]$,
	\end{enumerate}
\noindent and $w'$ is the $m$-weighting of $G$ given by
$$
w'(v)=\left\{
	\begin{array}{ll}
		0 & \textrm{if $v \in A$,}\\
		w(b_i) + w(a_i) \quad & \textrm{if $v = b_i$ for some $i \in
		[r]$,}\\
		w(v) & \textrm{otherwise}
	\end{array}\right.
$$
\noindent for each $v \in V(G)$, then $\pi_k(G(w')) \leq \pi_k(G(w))$.
\end{lemma}

\begin{proof}
Our strategy is to associate a unique $k$-clique of $G(w)$ to each $k$-clique of $G(w')$. More precisely, we construct an injective function from $\mathcal{K}_k(G(w'))$ to $\mathcal{K}_k(G(w))$. This gives us $\pi_k(G(w')) \leq \pi_k(G(w))$. 
	
Let $\phi : V(G(w'))~\rightarrow~V(G(w))$ such that, for each $(v,i) \in V(G(w'))$,  
	$$
	\phi((v,i)) = \left\{
		\begin{array}{ll}
			(a_j, i - w(b_j)) & \textrm{if $v = b_j$ and $i \in [w(b_j) + w(a_j)] \setminus [w(b_j)]$,} \\
			(v, i) & \textrm{otherwise.}
		\end{array}\right.
	$$
Note that $\phi$ is bijective. We may abbreviate $\phi((v,i))$ to $\phi(v,i)$. 

Consider any $X \in \mathcal{K}_k(G(w'))$. Thus, $G(w')[X]$ is a $k$-clique of $G(w')$. Let $(v_1,y_1)$, $\dots$, $(v_k,y_k)$ be the vertices in $X$. Let $\phi(X)$ denote $\{\phi(v_1,y_1),\ldots,\phi(v_k,y_k)\}$. Since $\phi$ is injective, $|\phi(X)|= k$. We show that $\phi(X) \in \mathcal{K}_k(G(w))$, that is, $\phi(v_i,y_i)\phi(v_j,y_j) \in E(G(w))$ for every $i, j \in [k]$ with $i \neq j$. If $\phi(v_i,y_i) = (v_i,y_i)$ and $\phi(v_j,y_j) = (v_j,y_j)$, then $y_i \in [w(v_i)]$, $y_j \in [w(v_j)]$, and, since $(v_i,y_i)(v_j,y_j) \in E(G(w'))$ (as $(v_i,y_i), (v_j,y_j) \in X \in \mathcal{K}_k(G(w'))$), $\phi(v_i,y_i)\phi(v_j,y_j) \in E(G(w))$. Suppose $\phi(v_i,y_i) \neq (v_i,y_i)$ or $\phi(v_j,y_j) \neq (v_j,y_j)$. We may assume that $\phi(v_i,y_i) \neq (v_i,y_i)$. Thus, $v_i = b_p$ for some $p \in [r]$, $y_i \notin [w(v_i)]$, and $\phi(v_i,y_i) = (a_p, y)$ for some $y \in [w(a_p)]$. If $v_j = b_p$, then $\phi(v_j,y_j) \in (\{a_p\} \times [w(a_p)]) \cup (\{b_p\} \times [w(b_p)])$, and hence, since $a_pb_p \in E(G)$, $\phi(v_i,y_i)\phi(v_j,y_j) \in E(G(w))$. Suppose $v_j \neq b_p$. Since $(v_i,y_i)(v_j,y_j) \in E(G(w'))$, $v_iv_j \in E(G)$. Thus, $v_j \notin B$ as $v_i \in B$ and $B$ is an independent set of $G$.  Since $(v_j,y_j) \in V(G(w'))$, we have $w'(v_j) \neq 0$, so $v_j \notin A$. Since $v_j \notin A \cup B$, we have $\phi(v_j,y_j) = (v_j,y_j)$, $y_j \in [w(v_j)]$, and hence $w(v_j) \neq 0$. Given that $N_G(b_p) \setminus \pa{A \cup \{v \in V(G) \colon w(v) = 0\}} \subseteq N_G(a_p)$, $v_j \in N_G(a_p)$ (as $b_pv_j = v_iv_j \in E(G)$). Thus, since $\phi(v_i,y_i)\phi(v_j,y_j) = (a_p, y)(v_j,y_j)$, $\phi(v_i,y_i)\phi(v_j,y_j) \in E(G(w))$.

Let $\Phi : \mathcal{K}_k(G(w')) \rightarrow \mathcal{K}_k(G(w))$ such that $\Phi(X) = \phi(X)$ for each $X \in \mathcal{K}_k(G(w'))$. Suppose $\Phi(X_1) = \Phi(X_2)$ for some $X_1, X_2 \in \mathcal{K}_k(G(w'))$. Let $(v_1,y_1),\ldots,(v_k,y_k)$ be the members of $X_2$. For each $i \in [k]$, let $(v_i',y_i')$ be the member of $X_1$ such that $\phi(v_i',y_i') = \phi(v_i,y_i)$. Since $\phi$ is injective, $(v_i',y_i') = (v_i,y_i)$ for each $i \in [k]$. Thus, $X_1 = X_2$. Therefore, $\Phi$ is injective, and hence the size of its domain $\mathcal{K}_k(G(w'))$ is at most the size of its codomain $\mathcal{K}_k(G(w))$.
\end{proof}

We remark that the condition in Lemma~\ref{lem:main} that $B$ is an independent set can be generalized to condition 2 in the next lemma, using the same argument. 

\begin{lemma}
If $m \geq 1$, $k \geq 2$, $w$ is an $m$-weighting of a graph $G$, $A = \{a_1,\ldots,a_r\}$ and $B = \{b_1,\ldots,b_r\}$ are disjoint $r$-element subsets of $V(G)$ such that

	\begin{enumerate}
		\item $a_i b_i \in E(G)$ for each $i \in [r]$,
		\item $a_ia_j, a_ib_j, a_jb_i \in E(G)$ for every $i, j \in [r]$ with $b_ib_j \in E(G)$, and 
		\item $N_G(b_i) \setminus \pa{A \cup \{v \in V(G) \colon w(v) = 0\}} \subseteq N_G(a_i)$ for each $i \in [r]$,
	\end{enumerate}

\noindent and $w'$ is the $m$-weighting of $G$ given by

$$
w'(v)=\left\{
	\begin{array}{ll}
		0 & \textrm{if $v \in A$,}\\
		w(b_i) + w(a_i) \quad & \textrm{if $v = b_i$ for some $i \in
		[r]$,}\\
		w(v) & \textrm{otherwise}
	\end{array}\right.
$$

\noindent for each $v \in V(G)$, then $\pi_k(G(w')) \leq \pi_k(G(w))$.
\end{lemma}

We now provide the proofs of Theorems~\ref{thm:cmg} and \ref{thm:chordal}, immediately demonstrating the applicability of Lemma~\ref{lem:main}. The lemma also has a crucial role in the proof of Theorem~\ref{thm:sperner_min}.

\begin{proof}[\textbf{Proof of Theorem~\ref{thm:cmg}}] Let $I_1, \dots, I_r$ be the distinct maximal partite sets of $G$ such that $|I_1| \geq \dots \geq |I_r|$. Note that $I_1, \dots, I_r$ are independent sets of $G$. For each $j \in [r]$, let $n_j = |I_j|$ and let $a_{j,1}, \dots, a_{j, n_j}$ be the distinct vertices in $I_j$. 

Let $W$ be the set of $m$-weightings of $G$. For each $w \in W$, let $h(w) = \max\{j \in [r] \colon w(v) \neq 0 \mbox{ for some } v \in I_j\}$. Let $W^* = \{w^* \in W \colon \pi_k(G(w^*)) \leq \pi_k(G(w)) \mbox{ for each } w \in W\}$. Let $w_0 \in W^*$ such that $h(w_0) \leq h(w)$ for each $w \in W^*$. 

Suppose $h(w_0) > 1$. Let $s = h(w_0)$. Let $J = \{a_{s-1,1}, \dots, a_{s-1,n_{s}}\}$. Since $J \subseteq I_{s-1}$, $J$ is an independent set of $G$. Since $G$ is a complete multipartite graph, $a_{s,i}a_{s-1,i} \in E(G)$ and $N_{G}(a_{s-1,i}) \setminus I_s \subseteq N_G(a_{s,i})$ for each $i \in [n_s]$. Let $w_1$ be the $m$-weighting of $G$ such that $w_1(a_{s,i})= 0$ for each $i \in [n_s]$, $w_1(a_{s-1,i}) = w_0(a_{s-1,i}) + w_0(a_{s,i})$ for each $i \in [n_s]$, and $w_1(v) = w_0(v)$ for each $v \in V(G) \setminus (I_s \cup J)$. By Lemma~\ref{lem:main}, $\pi_k(G(w_1)) \leq \pi_k(G(w_0))$. We have $h(w_1) = h(w_0) - 1$, contradicting the choice of $w_0$. Therefore, $s = 1$.

We have shown that $w_0(v) = 0$ for each $v \in V(G) \setminus I_1$. By Lemma~\ref{cor:cauchy}, $\pi_k(G(w')) \leq \pi_k(G(w_0))$. Consequently, $w' \in W^*$.
\end{proof}

\begin{proof}[\textbf{Proof of Theorem~\ref{thm:chordal}}] We use induction on $|V(G)| + |E(G)|$. Let $n = |V(G)|$. The result is trivial if $n = 1$ (the base case). Suppose $n \geq 2$. If $E(G) = \emptyset$, then the result is given by Lemma~\ref{lem:cauchy}. Suppose $E(G) \neq \emptyset$. Since $G$ is chordal, there exists a sequence $v_1, \dots, v_n$ such that $V(G) = \{v_1, \dots, v_n\}$ and $N_G[v_i] \setminus \{v_j \colon j \in [i-1]\} \in \mathcal{K}(G)$ for each $i \in [n]$. 

Suppose $N_G(v_1) = \emptyset$. Let $G_1 = G[V(G) \setminus \{v_1\}]$ and $m_1 = m - w(v_1)$. Let $w_1$ be the $m_1$-weighting of $G_1$ such that $w_1(v) = w(v)$ for each $v \in V(G_1)$. Let $w_1'$ be a uniform-$\alpha$ $m_1$-weighting of $G_1$. Then, for some largest independent set $I_1$ of $G_1$, $w_1'(v) = 0$ for each $v \in V(G_1) \setminus I_1$ and, for each $u \in I_1$, $w_1'(u) = \left\lfloor m_1/|I_1| \right\rfloor$ or $w_1'(u) = \left\lceil m_1/|I_1| \right\rceil$.  By the induction hypothesis, $\pi_k(G_1(w_1')) \leq \pi_k(G_1(w_1))$. Let $w''$ be the $m$-weighting of $G$ such that $w''(v_1) = w(v_1)$ and $w''(v) = w_1'(v)$ for each $v \in V(G_1)$. We have $\pi_k(G(w'')) = \binom{w(v_1)}{k} + \pi_k(G_1(w_1')) \leq \binom{w(v_1)}{k} + \pi_k(G_1(w_1)) = \pi_k(G(w))$. Since $N_G(v_1) = \emptyset$, $\{v_1\} \cup I_1$ is an independent set of $G$, so $\pi_k(G(w')) \leq \pi_k(G(w''))$ by Lemma~\ref{cor:cauchy}. Thus, $\pi_k(G(w')) \leq \pi_k(G(w))$. 


Now suppose $N_G(v_1) \neq \emptyset$. Let $h = \min\{i \in [n] \setminus \{1\} \colon v_i \in N_G(v_1)\}$. 
Let $w^*$ be the $m$-weighting of $G$ such that $w^*(v_1) = w(v_1) + w(v_h)$, $w^*(v_h) = 0$, and $w^*(v) = w(v)$ for each $v \in V(G) \setminus \{v_1, v_h\}$. By Lemma~\ref{lem:main}, $\pi_k(G(w^*)) \leq \pi_k(G(w))$. Let $G_2 = G[V(G) \setminus \{v_h\}]$. Let $w_2$ be the $m$-weighting of $G_2$ such that $w_2(v) = w^*(v)$ for each $v \in V(G_2)$. Let $w_2'$ be a uniform-$\alpha$ $m$-weighting of $G_2$. By the induction hypothesis, $\pi_k(G_2(w_2')) \leq \pi_k(G_2(w_2))$. Let $w''$ be the $m$-weighting of $G$ such that $w''(v_h) = 0$ and $w''(v) = w_2'(v)$ for each $v \in V(G_2)$. By Lemma~\ref{cor:cauchy}, $\pi_k(G(w')) \leq \pi_k(G(w''))$. Since $\pi_k(G(w'')) = \pi_k(G_2(w_2')) \leq \pi_k(G_2(w_2)) = \pi_k(G(w^*)) \leq \pi_k(G(w))$, it follows that $\pi_k(G(w')) \leq \pi_k(G(w))$.
\end{proof}


\section{Proofs of the main results}\label{sec:proofs}


In the proof of Theorem~\ref{thm:sperner_min}, we use the following well-known consequence of the K\H{o}nig--Hall Theorem \cite{Konig, Hall}; see, for example, \cite[page 7, Corollary 4]{Bollobas}.

\begin{lemma} \label{lemma:Hall} Let $r$ and $n$ be integers such that $0 \leq r \leq n$. If $r < n/2$, then there exists an injection $f : \binom{[n]}{r} \rightarrow \binom{[n]}{r+1}$ such that $A \subset f(A)$ for each $A \in \binom{[n]}{r}$. If $r > n/2$, then there exists an injection $f : \binom{[n]}{r} \rightarrow \binom{[n]}{r-1}$ such that $f(A) \subset A$ for each $A \in \binom{[n]}{r}$.
\end{lemma}

\begin{proof}[\textbf{Proof of Theorem~\ref{thm:sperner_min}}] Let $W$ be the set of $m$-weightings of $B_n$. For any $w \in W$, let $g(w) = \min\{i \in \{0\} \cup [n] \colon w(v) \neq 0 \mbox{ for some } v \in \binom{[n]}{i}\}$ and $g'(w) = \max\{i \in \{0\} \cup [n] \colon w(v) \neq 0 \mbox{ for some } v \in \binom{[n]}{i}\}$, and let
\begin{equation} h(w) = \min\{\lceil n/2 \rceil, \, g(w)\} \quad \mbox{ and } \quad h'(w) = \max\{\lceil n/2 \rceil, \, g'(w)\}. \nonumber
\end{equation}
%
Let $W^* = \{w^* \in W \colon \pi_k(G(w^*)) \leq \pi_k(G(w)) \mbox{ for each } w \in W\}$. Let $w_0 \in W^*$ such that $h'(w_0) - h(w_0) \leq h'(w) - h(w)$ for each $w \in W^*$. It suffices to show that $\pi_k(B_n(w')) \leq \pi_k(B_n(w_0))$.

Suppose $h(w_0) < \lceil n/2 \rceil$. Let $a_1, \dots, a_r$ be the distinct members of $\binom{[n]}{h(w_0)}$. By Lemma~\ref{lemma:Hall}, there exist $r$ distinct members $b_1, \dots, b_r$ of $\binom{[n]}{h(w_0)+1}$ such that $a_1b_1, \dots, a_rb_r \in E(B_n)$. Let $I = \{a_1, \dots, a_r\}$ and $J = \{b_1, \dots, b_r\}$. For each $i \in \{0\} \cup [n]$, $\binom{[n]}{i}$ is an independent set of $B_n$. Thus, $I$ and $J$ are independent sets of $B_n$. For any $i \in [r]$ and $c \in V(B_n)$ with $b_i \subset c$, we have $a_i \subset c$ as $a_i \subset b_i$. Thus, $N_{B_n}(b_i) \setminus (I \cup \{v \in V(B_n) \colon w_0(v) = 0\}) \subseteq N_{B_n}(a_i)$ for each $i \in [r]$. Let $w_1$ be the $m$-weighting of $B_n$ such that $w_1(a_i)= 0$ for each $i \in [r]$, $w_1(b_i) = w_0(b_i) + w_0(a_i)$ for each $i \in [r]$, and $w_1(v) = w_0(v)$ for each $v \in V(B_n) \setminus (I \cup J)$. By Lemma~\ref{lem:main}, $\pi_k(B_n(w_1)) \leq \pi_k(B_n(w_0))$. We have $h'(w_1) - h(w_1) = h'(w_0) - h(w_0) - 1$, contradicting the choice of $w_0$. Therefore, $h(w_0) = \lceil n/2 \rceil$.

Suppose $h'(w_0) > \lceil n/2 \rceil$. Let $a_1, \dots, a_r$ be the distinct members of $\binom{[n]}{h'(w_0)}$. By Lemma~\ref{lemma:Hall}, there exist $r$ distinct members $b_1, \dots, b_r$ of $\binom{[n]}{h'(w_0)-1}$ such that $a_1b_1, \dots, a_rb_r \in E(B_n)$. Let $I = \{a_1, \dots, a_r\}$ and $J = \{b_1, \dots, b_r\}$. For any $i \in [r]$ and $c \in V(B_n)$ with $c \subset b_i$, we have $c \subset a_i$ as $b_i \subset a_i$. Thus, $N_{B_n}(b_i) \setminus (I \cup \{v \in V(B_n) \colon w_0(v) = 0\}) \subseteq N_{B_n}(a_i)$ for each $i \in [r]$. Let $w_1$ be the $m$-weighting of $B_n$ such that $w_1(a_i)= 0$ for each $i \in [r]$, $w_1(b_i) = w_0(b_i) + w_0(a_i)$ for each $i \in [r]$, and $w_1(v) = w_0(v)$ for each $v \in V(B_n) \setminus (I \cup J)$. By Lemma~\ref{lem:main}, $\pi_k(B_n(w_1)) \leq \pi_k(B_n(w_0))$. We have $h'(w_1) - h(w_1) = h'(w_0) - h(w_0) - 1$, contradicting the choice of $w_0$. Therefore, $h'(w_0) = \lceil n/2 \rceil$.

We have shown that $w_0(v) = 0$ for each $v \in V(B_n) \setminus \binom{[n]}{\lceil n/2 \rceil}$. By Lemma~\ref{cor:cauchy}, $\pi_k(B_n(w')) \leq \pi_k(B_n(w_0))$. Consequently, $w' \in W^*$.
\end{proof}

We now show that, as stated in Remark~\ref{rem:sperner_min}, the answer to Problem~\ref{problem_naif} for $G = B_n$ is negative if $n \geq 2$, $k \geq 3$, and $m \geq k \alpha_k(B_n)$.

\begin{theorem} \label{rem1result}
If $n \geq 2$, $k \geq 3$, $m \geq k \alpha_k(B_n)$, $w$ and $w'$ are $m$-weightings of $B_n$, $w$ is uniform-$\alpha$, and $w'$ is uniform on a largest $k$-clique independent set of $G$, then $\pi_k(G(w)) < \pi_k(G(w'))$.
\end{theorem}
\begin{proof} We build on the proof of Theorem~\ref{thm:sperner_min}. Let $w_0$ be an $m$-weighting of $B_n$ that is uniform on a largest $k$-clique independent set $I$ of $B_n$. For a contradiction, suppose $w_0 \in W^*$. Since $m \geq k \alpha_k(B_n) = k |I|$, $w_0(v) \geq k$ for each $v \in I$. Let $M = \binom{[n]}{\ceil{n/2}}$ and $M' = \binom{[n]}{\ceil{n/2} - 1}$. Since $k \geq 3$ and $M \cup M'$ is a $3$-clique independent set of $B_n$, $|I| \geq |M| + |M'|$. 
Thus, $I \nsubseteq M$, and hence $h(w_0) < h'(w_0)$. 


Suppose $h(w_0) < \ceil{n/2}$. Let $p = h'(w_0) - \ceil{n/2}$. Thus, $p \geq 0$. 
If $p > 0$, then, as shown in the proof of Theorem~\ref{thm:sperner_min}, we can define a sequence $w_1, \dots, w_{p}$ of $m$-weightings of $B_n$ such that, for each $i \in [p]$, $h'(w_i) = h'(w_{i-1}) - 1$, $\pi_k(B_n(w_i)) \leq \pi_k(B_n(w_{i-1}))$, and hence $w_i \in W^*$. Note that $h'(w_p) = \ceil{n/2}$ and $h(w_p) = h(w_0)$. Let $q = p + \ceil{n/2} - h(w_0)$. Since $h(w_0) \leq \ceil{n/2} - 1$, $q \geq p+1$. 
As shown in the proof of Theorem~\ref{thm:sperner_min}, we can define a sequence $w_{p+1}, \dots, w_{q}$ of $m$-weightings of $B_n$ such that, for each $i \in [q-p]$, $h(w_{p+i}) = h(w_{p+i-1}) + 1$, $h'(w_{p+i}) = h'(w_p) = \ceil{n/2}$, $\pi_k(B_n(w_{p+i})) \leq \pi_k(B_n(w_{p+i-1}))$, and hence $w_{p+i} \in W^*$. Since $h(w_q) = \ceil{n/2} = h'(w_q)$, 
\begin{equation} \mbox{$w_q(v) = 0$ for each $v \in V(B_n) \setminus M$.} \label{zeroelsewhere}
\end{equation} 
Suppose $w_q(u) = 0$ for some $u \in M$. Since $m \geq k|I| > k |M|$, $w_q(u') > k$ for some $u' \in M$. Let $w_q'$ be the $m$-weighting of $B_n$ such that $w_q'(u) = 1$, $w_q'(u') = w_q(u') - 1$, and $w_q'(v) = w_q(v)$ for each $v \in V(B_n) \setminus \{u, u'\}$. This gives us $\pi_k(B_n(w_q')) < \pi_k(B_n(w_q))$, contradicting $w_q \in W^*$. Thus, $w_q(v) > 0$ for each $v \in M$. Note that, since $w_0(v) \geq k$ for each $v \in I$, and $w_0(v) = 0$ for each $v \in V(B_n) \setminus I$, we have 
$$\mbox{$w_i(v) \geq k$ for each $i \in [q]$ and each $v \in V(B_n)$ such that $w_i(v) \neq 0$.}$$
Let $a_1, \dots, a_r$ be the distinct members of $\{v \in M' \colon w_{q-1}(v) \neq 0\}$. There exist $r$ distinct elements $b_1, \dots, b_r$ of $M$ such that $a_ib_i \in E(B_n)$ and $w_q(b_i) = w_{q-1}(b_i) + w_{q-1}(a_i) \geq 2k$ for each $i \in [r]$, and we have $w_q(a_i) = 0$  and $w_{q-1}(a_i) \geq k$ for each $i \in [r]$, $w_q(v) = w_{q-1}(v) \geq k$ for each $v \in M \setminus \{b_1, \dots, b_r\}$, and $w_q(v) = w_{q-1}(v) = 0$ for each $v \in V(B_n) \setminus (M \cup \{a_1, \dots, a_r\})$. Let $X = \{a_1, \dots, a_r\}$, $Y = \{b_1, \dots, b_r\}$, and $Y' = \bigcup_{i=1}^r (N_{B_n}(a_i) \cap M)$. Thus, $Y \subseteq Y'$. Each member of $X$ is a set of size $\ceil{n/2} - 1$ and has exactly $n - (\ceil{n/2} - 1)$ supersets in $M$, and each member of $M$ has exactly $\ceil{n/2}$ subsets in $M'$. We have 
\begin{align} (n - \ceil{n/2} + 1)r &= \sum_{x \in X} |N_{B_n}(x) \cap M| = \{xy \in E(B_n) \colon x \in X, \, y \in Y'\} \nonumber \\
&= \sum_{y \in Y'} |N_{B_n}(y) \cap X| \leq \ceil{n/2}|Y'|, \nonumber
\end{align} 
so $r \leq |Y'|$, and equality holds only if $n$ is odd. Suppose $r < |Y'|$. Then, $y \notin Y$ for some $y \in Y'$. Now $y \in N_{B_n}(a_j)$ for some $j \in [r]$, so $a_jy \in E(G)$ and $y \neq b_j$. Let $S = \{A \in \mathcal{K}_k(B_n(w_{q-1})) \colon A \cap \{(a_j,i) \colon i \in [w_{q-1}(a_j)]\} \neq \emptyset \neq A \cap \{(y,i) \colon i \in [w_{q-1}(y)]\}\}$. Then, $S \neq \emptyset$ as $k \geq 3$, $w_{q-1}(a_j) \geq k$, and $w_{q-1}(y) \geq k$. We have
\begin{equation} \pi_k(B_n(w_{q-1})) \geq \sum_{v \in M \setminus Y} \binom{w_{q-1}(v)}{k} + \sum_{i=1}^r \binom{w_{q-1}(a_i) + w_{q-1}(b_i)}{k} + |S| = \pi_k(B_n(w_{q})) + |S|, \label{extraedgecliques}
\end{equation}
contradicting $w_{q-1} \in W^*$. Therefore, $r = |Y'|$, and hence $Y' = Y$ and $n$ is odd. Thus, no member of $X$ is a subset of a member of $M \setminus Y$, and hence $X \cup (M \setminus Y)$ is an antichain of size $|M|$. By Theorem~\ref{thm:sperner}, $X \cup (M \setminus Y)$ is $M$ or $M'$. Since $\emptyset \neq X \subseteq M'$, we obtain $X = M'$ and $Y = M$. Suppose $w_{q-1}(b_i) > 0$ for some $i \in [r]$. Since $n \geq 2$, $n$ is odd, and $X = M'$, we have $a_jb_i \in E(B_n)$ for some $j \in [r] \setminus \{i\}$. Let $y = b_i$ and define $S$ as above. As in (\ref{extraedgecliques}), we obtain $\pi_k(B_n(w_{q-1})) \geq \pi_k(B_n(w_{q})) + |S| > \pi_k(B_n(w_{q}))$, contradicting $w_{q-1} \in W^*$. Thus, since $M = Y = \{b_1, \dots, b_r\}$, $w_{q-1}(v) = 0$ for each $v \in V(B_n) \setminus M'$. This implies that $\max\{i \in \{0\} \cup [n] \colon w_0(v) \neq 0 \mbox{ for some } v \in \binom{[n]}{i}\} \leq \ceil{n/2}-1$, so $p = 0$ and, since $|I| > |M'|$, we have $h(w_0) \leq \ceil{n/2}-2$, $q \geq 2$, and $h(w_{q-2}) = \ceil{n/2}-2$. Let $M'' = \binom{[n]}{\ceil{n/2}-2}$. We apply the argument starting after (\ref{zeroelsewhere}) for $M'$, $M''$, $w_{q-1}$, and $w_{q-2}$ instead of $M$, $M'$, $w_{q}$, and $w_{q-1}$, respectively, and the inequality corresponding to $r \leq |Y'|$ that we obtain is strict. Similarly to the above, this contradicts $w_{q-2} \in W^*$.

Therefore, we must have $h(w_0) \geq \ceil{n/2}$. For any $v \in V(B_n)$, let $v' = [n] \setminus v$. Let $I' = \{v' \colon v \in I\}$, and let $w_0'$ be the $m$-weighting of $B_n$ such that $w_0'(v) = w_0(v')$ for each $v \in V(B_n)$. For any $u, v \in V(B_n)$, $u' \subseteq v'$ if and only if $v \subseteq u$, so $u'v' \in E(B_n)$ if and only if $uv \in E(B_n)$. Thus, $\pi_k(B_n(w_0')) = \pi_k(B_n(w_0))$. Recall that $h(w_0) < h'(w_0)$. We have $h(w_0') = n - h'(w_0) < \ceil{n/2}$ as $h(w_0) \geq \ceil{n/2}$. By applying the argument for $I$ to $I'$, we obtain a contradiction.

Therefore, $w_0 \notin W^*$. By Theorem~\ref{thm:sperner_min}, the result follows.
\end{proof}

We next show that, as stated in Remark~\ref{rem:cmg}, the answer to Problem~\ref{problem_naif} is also negative if $G$, $k$, and $m$ are as in Remark~\ref{rem:cmg}.

\begin{theorem} If $I_1, I_2, \dots, I_r$ are the distinct maximal partite sets of a complete multipartite graph $G$, $|I_1| > |I_2| \geq \dots \geq |I_r|$, $k \geq 3$, $m \geq k \alpha_k(G)$, $w$ and $w'$ are $m$-weightings of $G$, $w$ is uniform-$\alpha$, and $w'$ is uniform on a largest $k$-clique independent set of $G$, then $\pi_k(G(w)) < \pi_k(G(w'))$.
\end{theorem}

\begin{proof} We build on the proof of Theorem~\ref{thm:cmg}. We use the same idea in the proof of Theorem~\ref{rem1result}, which is to shift weights from one end to a largest independent set and shift weights from the other end to a second largest independent set. The proof for the current setting is similar but simpler. 

Let $M = I_1$ and $M' = I_2$. For each $w \in W$, let $h(w)$ be as in the proof of Theorem~\ref{thm:cmg}, and let $h'(w) = \min\{j \in [r] \colon w(v) \neq 0 \mbox{ for some } v \in I_j\}$. Let $w_0$ be an $m$-weighting of $G$ that is uniform on a largest $k$-clique independent set $I$ of $G$. Then, $|I| \geq |M| + |M'|$, so $h'(w_0) < h(w_0)$. Since $m \geq k \alpha_k(G) = k |I|$, $w_0(v) \geq k$ for each $v \in I$. For a contradiction, suppose $w_0 \in W^*$.

Let $p = h'(w_0) - 1$. If $p > 0$, then, as shown in the proof of Theorem~\ref{thm:cmg}, we can define a sequence $w_1, \dots, w_{p}$ of $m$-weightings of $G$ such that, for each $i \in [p]$, $h'(w_i) = h'(w_{i-1}) - 1$, $\pi_k(G(w_i)) \leq \pi_k(G(w_{i-1}))$, and hence $w_i \in W^*$. Note that $h'(w_p) = 1$ and $h(w_p) = h(w_0) > 1$. Let $q = p + h(w_0) - 1$. Since $h(w_0) \geq 2$, $q \geq p+1$. 
We can define a sequence $w_{p+1}, \dots, w_{q}$ of $m$-weightings of $G$ such that, for each $i \in [q-p]$, $h(w_{p+i}) = h(w_{p+i-1}) - 1$, $h'(w_{p+i}) = h'(w_p) = 1$, $\pi_k(G(w_{p+i})) \leq \pi_k(G(w_{p+i-1}))$, and hence $w_{p+i} \in W^*$. Since $h(w_q) = 1 = h'(w_q)$, 
\begin{equation} \mbox{$w_q(v) = 0$ for each $v \in V(G) \setminus M$.} \nonumber 
\end{equation} 
As in the proof of Theorem~\ref{rem1result}, $w_q(v) > 0$ for each $v \in M$, and
$$\mbox{$w_i(v) \geq k$ for each $i \in [q]$ and each $v \in V(G)$ such that $w_i(v) \neq 0$.}$$
Let $a_1, \dots, a_s$ be the distinct members of $\{v \in M' \colon w_{q-1}(v) \neq 0\}$. Let $a_{1,1}, \dots, a_{1,n_1}$, $\dots$, $a_{r,1}, \dots, a_{r,n_r}$ be as in the proof of Theorem~\ref{thm:cmg}. For each $i \in [s]$, $a_i = a_{2,j_i}$ for some $j_i \in [n_2]$. For each $i \in [s]$, let $b_i = a_{1,j_i}$. We have $w_q(b_i) = w_{q-1}(b_i) + w_{q-1}(a_i) \geq 2k$ for each $i \in [s]$. Let $Y = \{b_1, \dots, b_s\}$. Let $y = a_{1,n_1}$. Since $n_1 > n_2$, $y \notin Y$. As in the proof of Theorem~\ref{rem1result}, we obtain $\pi_k(G(w_{q-1})) > \pi_k(G(w_{q}))$ (as $y \in M \subseteq N_G(x)$ for each $x \in M'$), contradicting $w_{q-1} \in W^*$.

Therefore, $w_0 \notin W^*$. By Theorem~\ref{thm:cmg}, the result follows.
\end{proof}

\begin{proof}[\textbf{Proof of Theorem~\ref{thm:con_stable_false}}]
For any $n \in \mathbb{N}$, let $G_n$ be a triangle-free $n$-vertex graph and let $I_n$ be a largest independent set of $G_n$. Let $\alpha_n = |I_n|$, let $w_n$ be a $3n$-weighting of $G_n$ that is uniform on $V(G_n)$ (so $w(v) = 3$ for each $v \in V(G_n)$), and let $w_n'$ be a $3n$-weighting of $G_n$ that is uniform on $I_n$. It suffices to show that $\pi_3(G_n(w_n)) < \pi_3(G_n(w_n'))$ for some $n \in \mathbb{N}$. We have
$$\pi_3(G_n(w_n)) = \binom{3}{3}n + 2 \binom{3}{1} \binom{3}{2} \abs{E(G_n)} = n + 18 \abs{E(G_n)} $$
\noindent and
$$\pi_3(G_n(w_n')) \geq \binom{\partieentiere{\tf{3n}{\alpha_n}}}{3} \alpha_n. $$
Bohman \cite{Bohman} showed that we can choose the sequence $\{G_n\}_{n \in \mathbb{N}}$ so that $\alpha_n = {\rm O}(\sqrt{n \log n})$ and the maximum degree $\Delta(G_n)$ ($\max \{|N_{G_n}(v)| \colon v \in V(G_n)\}$) of $G_n$ satisfies $\Delta(G_n) = \Theta(\sqrt{n \log n})$ (see \cite[Theorem~2.5]{ABK}). Since $I_n$ is maximal, $V(G_n) = I_n \cup \bigcup_{v \in I_n} N_{G_n}(v)$, so 
$$n = |I_n \cup \bigcup_{v \in I_n} N_{G_n}(v)| \leq |I_n| + \sum_{v \in I_n} |N_{G_n}(v)| \leq (1 + \Delta(G_n)) \alpha_n,$$
and hence $\alpha_n \geq n / (1 + \Delta(G_n))$. Thus, for $n$ sufficiently large, we have 
$$\alpha_n \geq \frac{n}{a \sqrt{n \log n}}  \quad \mbox{ and } \quad \frac{1}{\alpha_n} \geq \frac{1}{b \sqrt{n \log n}}$$
\noindent for some positive real numbers $a$ and $b$, and hence
$$\pi_3(G_n(w_n')) \geq \binom{\partieentiere{\tf{3n}{b \sqrt{n \log n}}}}{3} \frac{n}{a \sqrt{n \log n}}.$$
Since $\binom{p}{3} \sim p^3/6$, we obtain
$$\binom{\partieentiere{\tf{3n}{b \sqrt{n \log n}}}}{3} = \Theta \paf{n^{\tf{3}{2}}}{\pa{\log n}^{\tf{3}{2}}}, $$
\noindent so $\pi_3(G_n(w_n')) = \Omega \pa{\tf{n^2}{(\log n)^2}}$. By the handshaking lemma, $2|E(G_n)| \leq n \Delta(G_n)$, so
$$\abs{E(G_n)} = \grando{}{\f{n}{2} \sqrt{n \log n}} = \grando{}{n^{3/2} \sqrt{\log n}}.$$
\noindent Thus, $\pi_3(G_n(w_n)) = \grando{}{n^{\tf{3}{2}} \sqrt{\log n}}$. Since $n^{\tf{3}{2}} \sqrt{\log n} = \petito{}{\tf{n^2}{(\log n)^2}}$, $\pi_3(G_n(w_n)) < \pi_3(G_n(w_n'))$ for $n$ sufficiently large.
\end{proof}


\section*{Acknowledgements}

The authors are grateful to the referees for checking the paper carefully and providing comments that led to an improvement in the presentation, and to St\'ephan Thomass\'e for a helpful discussion. Peter Borg was supported by grant MATRP14-22 of the University of Malta. Carl Feghali was supported by Agence Nationale de la Recherche (France) under research grant ANR DIGRAPHS ANR-19-CE48-0013-01.




\end{document}